\documentclass{article}

\usepackage{amsmath}
\usepackage{amsfonts}
\usepackage{amssymb}
\usepackage{array}
\usepackage{graphicx}
\usepackage{color}
\usepackage{multicol}
\usepackage{exercise}
\usepackage{amsthm}
\usepackage{tikz}
\usepackage{authblk}

\usepackage[none]{hyphenat}

\theoremstyle{plain}
  \newtheorem{thm}{Theorem}
  \newtheorem{prop}{Proposition}

\theoremstyle{remark}
  \newtheorem{example}{Example}
  \newtheorem{rem}{Remark}

\newcommand{\g}{\mathfrak{g}}

\newcommand{\R}{\mathbb{R}}

\newcommand{\CC}{\mathbb{C}}

\newcommand{\C}{\mathcal{C}}

\newcommand{\no}{\noindent}
\newcommand{\be}{\begin{equation}}

\newcommand{\ee}{\end{equation}}

\renewcommand{\phi}{\varphi}

\newcommand{\ep}{\varepsilon}

\title{Inequalities from Poisson brackets}

\author{A.~Alekseev\thanks{Anton.Alekseev@unige.ch, Universit\'e de Gen\`eve, 2-4 rue du Li\`evre, c.p. 64, 1211 Gen\`eve~4 (Switzerland)}, I.~Davydenkova \thanks{Irina.Davydenkova@unige.ch, Universit\'e de Gen\`eve, 2-4 rue du Li\`evre, c.p. 64, 1211 Gen\`eve~4 (Switzerland)}}

\begin{document}


\maketitle



\begin{abstract}
We introduce the notion of {\em tropicalization}  for Poisson structures on  $\mathbb{R}^n$ with coefficients in Laurent polynomials. To such a Poisson structure we  associate a polyhedral cone and  a constant Poisson bracket on this cone. There is a version of this formalism applicable to $\mathbb{C}^n$ viewed as a real Poisson manifold. In this case, the tropicalization gives rise to a completely  integrable system with action variables taking values in a polyhedral cone and angle variables spanning a torus.

As an example, we consider the canonical Poisson bracket on the dual Poisson-Lie group $G^*$ for $G=U(n)$ in the cluster coordinates of Fomin-Zelevinsky defined by  a certain choice of solid minors. We prove that the corresponding integrable system is isomorphic to the Gelfand-Zeitlin completely integrable system of Guillemin-Sternberg and Flaschka-Ratiu.
\end{abstract}

\section{Introduction}

Log-canonical coordinates on Poisson manifolds play an important role in Poisson Geometry. In particular, they have proved to be useful in the theory of cluster varieties (see {\em e.g.} \cite{GSW}). Log-canonical coordinates are characterized by the fact that for two coordinate functions, say $x$ and $y$, their Poisson bracket is of the form
$$
\{ x, y \} = c \, xy.
$$
If $x$ and $y$ take real positive values, one can define new coordinates  $\xi= \log(x)$ and
$\eta = \log(y)$ so as the Poisson bracket of $\xi$ and $\eta$ is constant,
$$
\{ \xi, \eta\} =c.
$$

In this paper, we consider Poisson brackets of more general type. 
For coordinate functions (that we denote again by $x$ and $y$) we now have
\be \label{eq:+laurent}
\{ x, y\} = c\, xy + p(x,y, \dots),
\ee
where $p(x,y, \dots)$ is a Laurent polynomial in $x,y$ and (possibly) other coordinate
functions. To a Poisson bracket of this type, we assign its {\em tropicalization} which is a 
pair $(\mathcal{C}, \{\cdot, \cdot\}_\infty)$ where $\C$ is a polyhedral cone
and $\{\cdot, \cdot\}_\infty$ is a constant Poisson bracket on $\C$.

Recall that the {\em tropical calculus} is a semi-ring structure on $\mathbb{R}$
where addition is replaced by the maximum function and multiplication is replaced by addition
$$
\xi \, +_{trop} \, \eta = \max(\xi,\eta), \hskip 0.3cm
\xi  \, \cdot_{trop} \, \eta= \xi+\eta.
$$
One can obtain this semi-ring structure as a $t \to + \infty$ limit of the standard 
semi-ring structure on $\mathbb{R}_+$ under the map $x \mapsto \xi =t^{-1} \log(x)$. Indeed,
$$
\lim_{t \to + \infty} \, t^{-1} \log \left( e^{t\xi} + e^{t\eta} \right) = \max(\xi, \eta), \hskip 0.3cm
\lim_{t \to + \infty} \, t^{-1} \log \left( e^{t\xi} \cdot e^{t\eta} \right) = \xi+ \eta.
$$
Returning to tropicalization of Poisson brackets, we consider an example
$$
\{ x, y\} = c\, xy + a\, x + b\, y.
$$
Let $t \in \mathbb{R}_+$ be
a real positive parameter, and let  $\xi= t^{-1} \log(x), \eta=t^{-1} \log(y)$. 
In coordinates $\xi,\eta$  the Poisson bracket acquires the form
$$
\{ \xi, \eta\}_t = t^{-2} \left( c + a e^{-t\eta} + be^{-t\xi}\right) .
$$
We require that the log-canonical contribution ($t^{-2} c$ on the right hand side) is dominant for $t \to + \infty$. This yields two inequalities
$$
\xi > 0, \,\, \eta>0
$$
which define the cone $\C$. By rescaling the bracket by a factor of $t^2$, we obtain an  expression which has a well-defined limit on $\C$ when $t$ tends to infinity,
$$
\{ \xi, \eta \}_\infty := \lim_{t \to + \infty} t^2 \{\xi, \eta\}_t = c.
$$
The resulting Poisson bracket $\{ \cdot, \cdot\}_\infty$ is constant.

There is a version of this formalism adapted to complex coordinate functions
$\{ z_1, \dots, z_n\}$ on a real Poisson manifold. In this case, we use the change
of variables $z_i = \exp(t \zeta_i + i \phi_i)$ with parameter $t \to + \infty$.
The result of the tropicalization procedure is again an open polyhedral cone $\C$
and a constant Poisson structure on $\C \times \mathbb{T}^n$. Under this constant
Poisson structure, coordinates $\zeta_i$ Poisson commute with each other. That is,
we obtain {\em a completely integrable system} with $\zeta_i$'s as action variables
and $\phi_i$'s as angle variables.

As an example, we consider the  Poisson bracket on the dual Poisson-Lie group
$G^*$ for $G=U(n)$. This Poisson bracket was defined in \cite{STS} and \cite{LW}.
As a coordinate system we use solid minors $\Delta^{(k)}_l$ from the total positivity
theory \cite{FZ}. Theorem of Kogan-Zelevinsky \cite{KZ} shows that these minors provide log-canonical coordinates on the Poisson-Lie group $G$. For the Poisson-Lie group
$G^*$, the corresponding Poisson bracket is no longer log-canonical, but it admits the form \eqref{eq:+laurent}.

The main result of this paper is the following theorem:
\begin{thm} \label{thm:intro}
The tropicalization of the Poisson bracket on the Poisson-Lie group $U(n)^*$
is isomorphic to the Gelfand-Zeitlin completely integrable system. 
%
\end{thm}

Our interest is motivated by the following observations: by the Ginzburg-Weinstein Isomorphism
Theorem \cite{GW}, the Poisson manifold $(G^*, \pi_{G^*})$ is isomorphic to $(\mathfrak{g}^*, \pi_{{\rm KKS}})$,
where $\pi_{{\rm KKS}}$ is the Kirillov-Kostant-Souriau Poisson bracket on $\mathfrak{g}^*$.
Since $\pi_{{\rm KKS}}$ is a linear Poisson structure, the scaling transformation $x \mapsto t^{-1} \, x,
\pi_t = t \, \pi_\text{KKS}$ is a Poisson isomorphism. This implies that $(G^*, \pi_{G^*})$ is 
Poisson isomorphic to $(G^*, t\pi_{G^*})$  for all $t >0$. 

The tropicalization procedure described in the paper assigns a limiting object at $t=+\infty$ to the family 
$(G^*, t\pi_{G^*}).$  Theorem \ref{thm:intro} shows that in the case of $G=U(n)$ this object is isomorphic to the Gelfand-Zeitlin completely integrable system.  Flaschka-Ratiu \cite{FR} discovered a Gelfand-Zeitlin type integrable system on $G^*$, and in \cite{AM} it was shown that the Flaschka-Ratiu system is isomorphic to the Gelfand-Zeitlin system. Hence, Theorem \ref{thm:intro} provides a $t=+\infty$ extension of the Ginzburg-Weinstein Isomorphism in the case of $G=U(n)$.

The structure of the paper is as follows: in Section \ref{logtrop} we introduce a notion of \textit{tropicalization} and associate a polyhedral cone and a constant Poisson bracket on this cone to a certain type of Poisson structures. First, we consider real positive Poisson manifolds, then we allow for complex coordinate functions and introduce a notion of \textit{linear scaling}. In Section~\ref{poisBG0} we consider Poisson structures on the group of upper triangular matrices and on its close relative $G_0^{*}.$ Finally, in Section~\ref{PoisG} we apply the machinery developed in Section~\ref{logtrop} to the Poisson structure on the dual Poisson-Lie group $U(n)^*$ to obtain the isomorphism with the Gelfand-Zeitlin completely integrable system.

\vskip 0.3cm

{\bf Acknowledgements.} We are grateful to M. Podkopaeva and A. Szenes
for useful discussions. We are indebted to the referees of this paper for their
valuable remarks and comments.

Our research was supported in part by the grant MODFLAT
of the European Research Council and by the NCCR SwissMAP of the Swiss National Science Foundation. Research of A.A. was supported in part  by the grants 200020--140985 and 200020--141329
of the Swiss National Science Foundation.
Research of I.D. was supported in part by the grant PDFMP2--141756 of the
Swiss National Science Foundation. 

\section{Log-canonical Poisson brackets \newline and Tropicalization}   \label{logtrop}

\subsection{Real positive Poisson manifolds}


Let $M$ be a real Poisson manifold and $U \subset M$ be a coordinate chart
with positive coordinate functions $\{ x_1, x_2, \dots, x_N\}, \,\, x_i \in \R_+$.

We say that the Poisson bivector
$\pi$ on $M$ is log-canonical with respect to the coordinate chart $U$ if it has the form
$$
\pi= \frac{1}{2} \, \sum_{i,j} \, \pi_{i,j} x_i x_j \, \frac{\partial}{\partial x_i} \wedge \frac{\partial}{\partial x_j}.
$$
That is, the Poisson brackets of coordinate functions are given by formula
$$
\{ x_i, x_j\} = \pi_{i,j} x_ix_j ,
$$
where no summation over repeating indices is assumed. 

The main object of our study will be Poisson brackets of the form
\begin{equation} \label{eq:pij}
\{ x_i, x_j\} = \pi_{i,j} x_ix_j + p_{i,j}(x),
\end{equation}
where $p_{i,j}(x)$ are Laurent polynomials in variables $x_1, \dots, x_N$.

The procedure of {\em tropicalization} will associate two combinatorial objects to a Poisson bracket of type \eqref{eq:pij}: an open polyhedral cone $\C(\pi; x)$ and a constant Poisson bracket on this cone. 

Let $V=\R^N$ with elements $(\xi_1, \dots, \xi_N) \in V$ and let $\{ e_1, \dots, e_N\}$ be the corresponding dual basis in $V^*$. For every pair $(i,j)$ with $1 \leq i < j \leq N$ consider the decomposition
$$
p_{i,j} = \sum_{I \in F_{i,j}} c_I x^I,
$$
where $I=(i_1, \dots, i_N)$ is a multi-index, $x^I=x_1^{i_1} \dots x_N^{i_N}$,
and $F_{i,j}$ is the set of multi-indices for which the coefficients $c_I$ are non-vanishing.
Put $n_{i,j}=e_i + e_j \in V^*$, denote
$$
n(I) = \sum_{r=1}^N i_r e_r
$$
%
and let $\C_{i,j} \subset V$ be the convex cone defined as follows
$$
\C_{i,j} = \{ \xi \in V; \,\, \langle n_{i,j} - n(I) , \xi \rangle > 0 \,\, \forall \,\, I \in F_{i,j}\}.
$$

\no In more detail, the cone $\C_{i,j}$ is defined by the inequalities
$$
\xi_i + \xi_j > \sum_{r=1}^N i_r \xi_r
$$
for all $I \in F_{i,j}$.
We define the cone $\C(\pi; x) \subset V$ as the intersection of 
the cones $\C_{i,j}$ for all pairs $(i,j)$:
$$
\C(\pi; x) = \cap_{i<j} \,\, \C_{i,j}.
$$

\begin{example}
Let $\pi$ be a log-canonical Poisson bracket in coordinates $x_1, \dots, x_N$. Then
the set $F_{i,j}$ is empty for all $i,j$ and $\C_{i,j} = V$ which yields $\C(\pi; x) = V$. 
\end{example}

\begin{example}
Let $N=2$ and let $\pi$ be the Poisson bracket defined by formula
$$
\{ x_1, x_2\} = x_1(1+x_2^2).
$$
In this case, we obtain two inequalities,
$$
\xi_1 + \xi_2 > \xi_1, \hskip 0.3cm 
\xi_1 + \xi_2 > \xi_1 + 2 \xi_2.
$$
They contradict each other, and in this case the cone $\C(\pi, x)$ is empty.
\end{example}

The next step is to introduce the following {\em scaling transformation}: let $t \in \mathbb{R}_+$ be a parameter, make a change of variables $x_i = \exp(t\xi_i),
\xi_i= \frac{1}{t} \ln(x_i)$ and scale the Poisson bivector as follows, $\pi_t = t^2 \pi$. If the Poisson bracket is log-canonical, it will become constant in variables $\xi_1, \dots, \xi_N$
$$
\{ \xi_i, \xi_j \}_t = \frac{1}{t^2} \, \{ \ln(x_i), \ln(x_j)\}_t = \pi_{i,j}.
$$
Note that the right hand side does not depend on $t$. This observation motivates the following definition:
let $\pi$ be a Poisson bracket of the form \eqref{eq:pij}. Then, a constant Poisson
bracket on the cone $\C(\pi; x)$ denoted by $\pi_\infty$ and given by the following formula can be associated to it
$$
\pi_\infty = \frac{1}{2} \, \sum_{i,j} \, \pi_{i,j} \, \frac{\partial}{\partial \xi_i} \wedge \frac{\partial}{\partial \xi_j}
$$
thus $\{ \xi_i, \xi_j\}_\infty=\pi_{i,j}$.

\begin{example}
\label{ex:abc}
Let $N=2$ and 
$$
\{ x_1, x_2\} = x_1x_2 + x_1^2+x_2
$$
\no which implies $F_{1,2}=\{(2,0);(0,1)\}$. The set $F_{1,2}$, the vector $n_{1,2}$ and
the cone $ \C$ are represented on the Figure \ref{ex2d}.

\begin{figure}[!htbp]
\label{ex2d}
\begin{center}
\includegraphics[width=8cm]{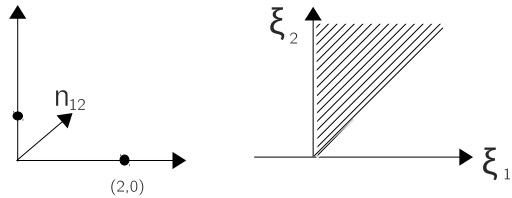}
\caption{The set $F_{1,2}$ and the cone $\C$  }
\end{center}
\end{figure}

\no It is easy to see that the Poisson bracket $\pi_\infty$ is of the form $ \{ \xi_1, \xi_2 \}_{\infty} = 1$.

\end{example}

\begin{prop}
Let $\pi$ be a Poisson bracket of type \eqref{eq:pij}. Then, in coordinates $\xi$ we have 
$$
t^2 \pi \to_{t \to + \infty} \pi_\infty.
$$
\no  for $\xi$ in $ \C(\pi; x).$
\end{prop}

\begin{proof}
Let $1 \leq i < j \leq N$, and compute
$$
\{ x_i, x_j \}_t = \{ e^{t\xi_i}, e^{t\xi_j} \}_t = t^2 e^{t(\xi_i + \xi_j)} \, \{\xi_i, \xi_j\}_t .
$$
That is, for the bracket $\{\xi_i, \xi_j\}_t$ we obtain the following expression

\begin{multline*}
\{\xi_i, \xi_j\}_t = t^{-2} e^{- t(\xi_i + \xi_j)} \{ x_i, x_j \}_t =\\
=e^{- t(\xi_i + \xi_j)}\left( \pi_{i,j} e^{t(\xi_i + \xi_j)}
+ \sum_{I \in F_{i,j}} c_I  e^{t \sum i_k \xi_k}\right) .
\end{multline*}
For $\xi \in \C(\pi; x)$ we have $\xi_i + \xi_j > \sum i_k \xi_k$ for all $I \in F_{i,j}$.
Hence, the right hand side tends to $\pi_{i,j}$ when $t \to +\infty$.
\end{proof}

\subsection{Complex coordinates and linear scaling}
In this Section, we shall allow for complex valued coordinate functions.
The coordinate chart $U$ will carry coordinates of the form
$\{ x_1, \dots, x_k, z_1, \dots, z_l\}$, where $x_1, \dots, x_k$ are real positive and $z_1, \dots, z_l$ are complex valued non-vanishing functions. Then, the real dimension of $M$ is $2l+k$, and we also get complex conjugates of the coordinate functions $\bar{z}_1, \dots, \bar{z}_l$ on $U$. 

A Poisson bracket $\pi$ is log-canonical in the coordinate chart $U$ is it is of the form
%
$$
\begin{array}{lll}
\{ x_i, x_j\} = \pi_{i,j} x_i x_j , &
\{ x_i, z_a\} = \pi_{i,a} x_i z_a , &
\{ x_i, \bar{z}_a\} =   \pi_{i,\bar{a}} x_i \bar{z}_a,   \\
\{ z_a, z_b \} = \pi_{a,b} z_a z_b, &
\{ \bar{z}_a, \bar{z}_b \} = \pi_{\bar{a}, \bar{b}} \bar{z}_a \bar{z}_b, &
\{z_a, \bar{z}_b\} = \pi_{a, \bar{b}} z_a \bar{z}_b.
\end{array}
$$
Since the bivector $\pi$ is supposed to be real, we have the following reality conditions imposed
on the components of $\pi$:
$$
\pi_{i, \bar{a}} = \overline{\pi_{i,a}} \hskip 0.2cm , \hskip 0.2cm
\pi_{\bar{a}, \bar{b}} = \overline{\pi_{a,b}} \hskip 0.2cm , \hskip 0.2cm
\pi_{a, \bar{b}} = - \overline{\pi_{b, \bar{a}}} \hskip 0.2cm .
$$

\begin{rem}
A more conceptual way to introduce log-canonical Poisson structures is as follows: let $G=\mathbb{R}_+^k \times (\mathbb{C}^*)^l$ be an abelian real Lie group with point-wise multiplication. Then, log-canonical Poisson structures are exactly the translation-invariant Poisson structures on $G$ (since $G$ is abelian, left and right translations coincide).\footnote{We are grateful to the referee for this remark.}

\end{rem}

More generally, we shall consider Poisson brackets of the form

\be \pi = \pi_0 + \pi', \label{brcom} \ee 

\no where $\pi_0$ is a log-canonical Poisson bracket and
$\pi'$ is a bivector with coefficients in Laurent polynomials in variables
$x, z$ and $\bar{z}$. Let $V = \R^{k+l}$ with elements
$(\xi_1, \dots, \xi_k, \zeta_1, \dots, \zeta_l)$. Denote the dual basis in
$V^*$ by $e_i, i=1, \dots, k$ and $f_a, a=1, \dots, l$.
Similarly to the previous Section, we define the cones $\C_{i,j}$ for $i<j$,
$\C_{a,b}$ for $a<b$ and $\C_{i, a}$. For example, we have
$$
\{ x_i, z_a\} = \pi_{i,a} x_i z_a + p_{i,a}(x,z),
$$
where $p_{i,a}(x,z)$ is a Laurent polynomial in variables $x,z$ and $\bar{z}$.
It can be written in the form
$$
p_{i,a}(x,z) = \sum_{I,J,K \in F_{i,a}} \,\, c_{I,J,K} x^I z^J \bar{z}^K,
$$
where $I,J$ and $K$ are multi-indices, and $F_{i,a}$ is the finite set
where the coefficients $c_{I,J,K}$ are non-vanishing. Denote
$n_{i,a}=e_i + f_a \in V^*$ and 
$$
n(I, J, K)= \sum_{r=1}^k i_r e_r + \sum_{s=1}^l (j_s + k_s) f_s
$$
for $(I,J,K) \in G_{i,a}$. The cone $\C_{i,a}$ is defined as follows
$$
\C_{i,a}= \{ \eta=(\xi, \zeta) \in V; \,\, \langle n_{i,a} - n(I,J,K), \eta\rangle >0 \,\, \forall I,J,K \in F_{i,a}\},
$$
\no That is we have the inequalities 

$$
\xi_i + \zeta_a > \sum_{r=1}^k i_r \xi_r+\sum_{s=1}^l (j_s +k_s)\zeta_s.
$$

\no We define the cone $\C(\pi; x, z)$ as the intersection of the cones
$\C_{i,j}, \C_{i,a}$ and $\C_{a,b}$.

We shall assume in addition the following reality conditions on the 
log-canonical part of the bivector $\pi$:
\be \label{reality}
\begin{array}{llll}
 \pi_{i,j}=0, & {\rm Re} \, \pi_{i,a}=0,&
 {\rm Re} \, \pi_{a,b} = 0, & {\rm Re} \, \pi_{a, \bar{b}}=0.
\end{array}
\ee
Under these assumptions, a log-canonical bivector admits the following linear
scaling. Again, let $t \in \mathbb{R}_+$ be a parameter. We introduce new coordinates
on $U$ via $x_i=\exp(t\xi_i), z_a=\exp(t\zeta_a + i \phi_a)$. Consider the scaled
Poisson bracket $\pi_t = t \pi$ in new coordinates. It yields the following Poisson brackets:
\be \label{brs}
\begin{array}{ll}
\{ \xi_i, \xi_j\}_t = 0, &
\{ \xi_i, \zeta_a\}_t = 0  , \\
\{ \xi_i, \phi_a\}_t =   {\rm Im} \, \pi_{i,a} ,  &
\{ \zeta_a, \zeta_b \}_t =  0 , \\
\{ \zeta_a, \phi_b \}_t =   \frac{1}{2} \, {\rm Im} \, (\pi_{a,b} - \pi_{a, \bar{b}}), &
\{\phi_a, \phi_b\}_t = 0 .
\end{array}
\ee

As before, this bracket does not depend on $t$, and we can denote it by $\pi_\infty$.
It is defined on the product $\C(\pi; x, z) \times \mathbb{T}^l$, where
$(\xi, \zeta) \in \C(\pi; x,z)$ and $(\phi_1, \dots, \phi_l) \in \mathbb{T}^l$,
the real torus of dimension $l$.

\begin{rem}
Log-canonical Poisson brackets without reality conditions \eqref{reality} do not
allow for a linear scaling limit. Instead, one can consider the limit of $\pi$ (no
powers of $t$ added) in coordinates $(\xi, \zeta, \phi)$. It yields constant
Poisson brackets between the angle variables $\{ \phi_a, \phi_b\}$ while
$\xi$'s and $\zeta$'s become Casimir functions in the limit.
\end{rem}

\begin{rem}
Log-canonical Poisson brackets with reality condition \eqref{reality} admit the following
geometric interpretation. Consider the manifold $G=\mathbb{R}_+^k \times (\mathbb{C}^*)^l$ 
as a graded manifold with the base $\mathbb{T}^l=(S^1)^l$ the real torus of dimension $l$
parametrized by the angles $\phi_a= {\rm Arg}(z_a), a=1, \dots, l$. These angle coordinates 
have degree zero. Declare the coordinates $\xi_i = \log(x_i)$ and $\zeta_a = \log(|z_a|)$ to 
be of degree 1. Then, conditions \eqref{reality} are equivalent to saying that the Poisson
structure is of degree one.
\end{rem}

\begin{rem}
Note that  log-canonical Poisson brackets with reality conditions
\eqref{reality} naturally give rise to completely integrable systems.
Indeed, variables $\xi_i$ and $\zeta_a$ Poisson commute.
Assuming that the rank of the bracket $\pi$ is equal to $2l$ (which is 
the maximal possible rank), this is a maximal family of Poisson 
commuting functions. The dual angles are $\phi_a$'s. They are
spanning the Liouville tori. The variables $(\xi, \zeta, \phi)$ are in fact 
action-angle variables for the resulting completely integrable system.
\end{rem}

\begin{example}
Let $k=1, l=1$ and consider the Poisson bracket of the form
$$
\{ x, z \} = ixz, \indent
\{ x, \bar{z}\} = -i x \bar{z}, \indent
\{ z, \bar{z}\} =i(x^2-x^{-2}).
$$
The set $G_{1,\bar{1}}$ and the cone $\C$ are represented at the Figure \ref{exXZ}:
\begin{figure}[!htbp]
\label{exXZ}
\begin{center}
\includegraphics[width=10cm]{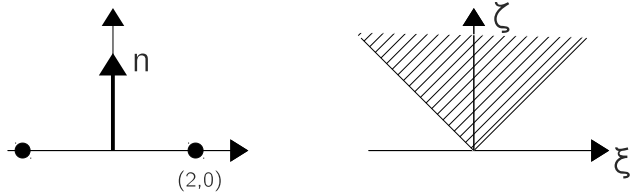}
\caption{The set $F_{1,\bar{1}}$ and the cone $\C$  }
\end{center}
\end{figure}

\no After changing variables $x=e^{t\xi},\, z=e^{t\zeta +i\phi}, \bar{z}=e^{t\zeta-i\phi} $ 
and applying the $t \to +\infty$ limit we obtain the 
following constant Poisson bracket on $\C\times\mathbb{S}^1$
$$
\{ \xi, \zeta \}_\infty =0 , \indent
\{ \xi, \phi \}_\infty =1, \indent
\{ \zeta, \phi \}_\infty =0.
$$
\end{example}

\begin{prop}
Let $\pi$ be a Poisson bracket of type (\ref{brcom}) verifying reality conditions \eqref{reality}. Then,  in coordinates $(\xi,\zeta,\phi)$ 
we have
$$
t \pi \to_{t \to + \infty} \pi_\infty.
$$
for $(\xi,\zeta,\phi) \in \C(\pi; x, z)\times \mathbb{T}^l$
\end{prop}
\begin{proof}
 
\no We shall give a proof for the Poisson bracket
 $\{ \xi_i, \phi_a\}_\infty$, the calculation for other Poisson brackets between coordinates
 is similar and will be omitted. Consider
\begin{multline*}
\{ x_i, z_a\}_t = \{e^{t\xi_i},e^{t\zeta_a+i\phi_a}\}_t=\\
=t^2e^{t(\xi_i+\zeta_a)+i\phi_a}\{\xi_i,\zeta_a\}_t+ite^{t(\xi_i+\zeta_a)+i\phi_a}\{\xi_i,\phi_a\}_t .
\end{multline*}

\no Thus, for the bracket $\{ \xi_i, \phi_a\}_t$ we obtain the following expression

\begin{multline*}
\{ \xi_i, \phi_a\}_t = t^{-1}e^{-t(\xi_i+\zeta_a)}{\rm Im}\left[ e^{-i\phi_a} \{x_i,z_a\}_t \right]=\\
=t^{-1}e^{-t(\xi_i+\zeta_a)}{\rm Im}\left[te^{-i\phi_a}(\pi_{i,a}e^{t\xi_i}e^{t\zeta_a+i\phi_a}) +te^{-i\phi_a}p_{i,a}(x,z) \right]=\\
={\rm Im}\,\left[ \pi_{i,a}+e^{-t(\xi_i+\zeta_a)}e^{-i\phi_a}\sum_{I,J,K \in G_{i,a}} \,\, c_{I,J,K} x^I z^J \bar{z}^K \right] .
\end{multline*}

\no 
Let $I,J,K \in G_{i,a}$ and consider the expression 
$$
x^I z^J \bar{z}^K = \exp\left( \sum_{r=1}^k i_r t\xi_r + \sum_{s=1}^l j_s (t\zeta_s + i \phi_s) + 
\sum_{t=1}^l k_t (t\zeta_t - i \phi_t)\right) .
$$

\no For $(\xi,\zeta) \in \C(\pi; x,z)$, we have 
$$
\xi_i + \zeta_a > \sum_{r=1}^k i_r \xi_r + \sum_{s=1}^l (j_s + k_s) \zeta_s
$$
for all  $I,J,K \in G_{i,a}$. Hence, the exponential $e^{t(\xi_i+\zeta_a)}$ dominates all the 
expressions $c_{I,J,K} x^I z^J \bar{z}^K$ and $e^{-t(\xi_i+\zeta_a)} c_{I,J,K} x^I z^J \bar{z}^K$
tends to zero when $t \to + \infty$, as required.
\end{proof}

\section{Poisson brackets on Poisson-Lie groups $B_+$ and $G^*_0$} \label{poisBG0}

In this Section we recall the definitions of Poisson brackets and of log-canonical coordinates
on the group of upper triangular invertible matrices and on its close relative the group $G^*_0$.

\subsection{Poisson-Lie group of upper triangular matrices}
Let $\g={\rm gl}(n, \mathbb{C})$, and let $r \in \g \otimes \g$ be the standard classical
$r$-matrix given by formula
$$
r = \frac{1}{2} \, \sum_i \, e_{i,i} \otimes e_{i,i} + \sum_{i<j} e_{i,j} \otimes e_{j,i},
$$
where $e_{i,j}$ is the elementary matrix with the only non-vanishing matrix entry 
equal to $1$ at the intersection of the $i$'th row and $j$'th column. Sometimes it is
convenient to split the $r$-matrix into two parts,
$$
r_0 = \frac{1}{2} \, \sum_i \, e_{i,i} \otimes e_{i,i} 
\hskip 0.3cm , \hskip 0.3cm
r' = \sum_{i<j} e_{i,j} \otimes e_{j,i} .
$$

The group $B_+$
of invertible upper-triangular matrices carries a Poisson structure given by formula
\be   \label{rgg-ggr}
\{ g^1 , g^2\} = [r, g^1g^2 ] = r g^1g^2 - g^1g^2 r,
\ee
where we are using the Saint--Petersburg notation $g^1 = g \otimes 1, g^2= 1 \otimes g$.

\begin{rem}
To illustrate the usage of this notation convention, consider a simpler bracket
$\{g^1, g^2\} = r_0 g^1 g^2$. In terms of more standard notation, this bracket looks as
$$
\{ f, h\} = \langle \nabla^L f \otimes \nabla^L h, r_0\rangle,
$$
where $f$ and $h$ are two functions on $B_+$, $\nabla^L$ is defined as
$$
(\nabla^L f)_g(x) = \frac{d}{dt} \, f\big( e^{tx} g) |_{t=0}
$$
for $x \in \mathfrak{b}_+  = {\rm Lie}(B_+)$, and  the pairing $\langle \cdot, \cdot \rangle$ is
induced by the natural pairing between $\mathfrak{b}_+$ and $\mathfrak{b}_+^*$.

The writing $\{g^1, g^2\} = r_0 g^1 g^2$ encodes the following (non skew-symmetric)
brackets of the  matrix elements of $g$:
$$
\{ g_{ij}, g_{st} \} = \delta_{is} g_{ij} g_{st}.
$$
This formula is obtained by taking the matrix element $(i,j)$ in the first factor of the tensor 
product (the matrix $g^1$), and the matrix element $(s,t)$ in the second factor
(the matrix $g^2$). Note that the brackets of matrix elements completely determine the Poisson
bracket on $B_+$.
\end{rem}

The Jacobi identity for the bracket \eqref{rgg-ggr} is a corollary of the classical Yang-Baxter equation
for the element $r$:
\begin{equation}
[r^{1,2}, r^{2,3}] + [r^{1,2}, r^{1,3}] + [r^{1,3}, r^{2,3}] =0.
\end{equation}
The group multiplication $B_+ \times B_+ \to B_+$ is a Poisson map making $B_+$
into a Poisson-Lie group.

Following Kogan-Zelevinsky \cite{KZ}, we introduce a log-canonical coordinate 
chart on $B_+$ in the following way. Let $ n \geq k \geq l \geq 1$ and denote by
$\Delta^{(k)}_l$ the solid minor of the matrix $g$ of size $l$ formed by the intersection of 
rows with consecutive numbers $n-k+1, \dots, n-k+l$ and the last $l$ columns (see Figure \ref{uptrmin}).
These $n(n+1)/2$ minors define coordinates on an open dense subset in $B_+$.
Hence, a smooth Poisson bracket on $B_+$ is completely characterized by the
brackets between $\Delta^{(k)}_l$'s.

\begin{figure}[!htbp]
\label{uptrmin}
\begin{center}
\includegraphics[width=7cm]{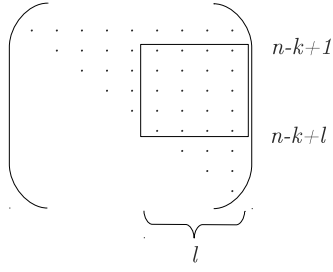}
\caption{minor $\Delta^{(k)}_l$ }
\end{center}
\end{figure}

\begin{thm}   \label{thm:DeltaPoisson}
The Poisson bracket of functions $\Delta^{(k)}_l, n\geq k \geq l \geq 1$ is log-canonical, and it has the
form
\begin{equation}  \label{eq:DeltaPoisson}
\{ \Delta^{(k)}_l, \Delta^{(p)}_q \} = \frac{1}{2} \epsilon(k-p) ( C - R ) \Delta^{(k)}_l \Delta^{(p)}_q,
\end{equation}
where $R$ is the number of common rows, $C$ is the  number of common columns of the 
two minors, and $\epsilon(x)$ is the sign function (that is, $\epsilon(x)=1$ for $x>0$,
$\epsilon(x)=-1$ for $x<0$, and $\epsilon(0)=0$).
\end{thm}

\begin{proof}
To prove the theorem we shall use formula \eqref{bracketdetailed} (see Appendix \ref{appA})
for a Poisson bracket  of two arbitrary minors $g_{IJ}$ and $g_{KL}$ which reads
\begin{multline*}
\{ g_{IJ}, g_{ST} \}
=\sum_{u<v}\chi_I(u) \chi_S(v) \, g_{\sigma_{u,v}(I), J} g_{\sigma_{v,u}(S), T}-\\
-\sum_{u<v}\chi_J(v) \chi_T(u) \, g_{I,\sigma_{v,u}(J)} g_{S,\sigma_{u,v}(T)}
 + \frac{1}{2} \, (| I \cap S| -| J \cap T |)\, g_{IJ}  g_{ST} ,
\end{multline*}
where $\chi_I$ is the characteristic function of the set $I$ (that is,
$\chi_I(k)=1$ for $k \in I$ and $\chi_I(k)=0$ for $k \notin I$), and
$\sigma_{v,u}(I)$ is the set obtained from $I$ by replacing $v$ with $u$.

Consider the second term on the right hand side.
In our situation, either $J \subset T$ or $T \subset J$ (or $J=T$). 
Hence, one of these subsets necessarily contains both $u$ and $v$.
After the replacement the corresponding matrix will contain two
identical columns and its determinant (either  $g_{I,\sigma_{v,u}(J)}$
or $g_{S,\sigma_{u,v}(T)}$) will vanish. Therefore,
this term always vanishes.

In the first term on the right hand side, non trivial contributions
come from the terms with $u \in I \backslash (I \cap S)$ and
$v \in S \backslash (I \cap S)$. If $p\geq k$, this implies
$v<u$ whereas the summation is over the range of $u<v$.
Hence, in this case the first term in the sum  vanishes as well.

By definition, $R= | I \cap S|$ and $C= |J \cap T|$ which yields
for $k\geq p$
$$
\{ \Delta^{(k)}_l, \Delta^{(p)}_q \} = \frac{1}{2}  ( C - R ) \Delta^{(k)}_l \Delta^{(p)}_q.
$$
The statement of the theorem follows by skew-symmetry of the Poisson bracket.
\end{proof}

\begin{example}
Let $n=2$. In this case, for 
$$
g=
\left(
\begin{array}{ll}
g_{11}  & g_{12} \\
0 & g_{22}
\end{array}
\right)
$$
we have three coordinate functions on $B_+$
$$
\Delta^{(1)}_1 = g_{22}, \hskip 0.3cm
\Delta^{(2)}_1 = g_{12}, \hskip 0.3cm
\Delta^{(2)}_2 = g_{11} g_{22}.
$$
Their Poisson brackets read
$$
\{ \Delta^{(1)}_1, \Delta^{(2)}_2\} = 0, \hskip 0.3cm
\{ \Delta^{(2)}_1, \Delta^{(2)}_2\} =0, \hskip 0.3cm
\{ \Delta^{(1)}_1, \Delta^{(2)}_1\} = -\frac{1}{2} \Delta^{(1)}_1 \Delta^{(2)}_1.
$$
Note that the determinant of $g$ is a Casimir function. Putting
$\Delta^{(2)}_2=1$, we obtain a Poisson algebra with generators
$\Delta^{(1)}_1$ and $\Delta^{(2)}_1$ and the log-canonical Poisson bracket
$\{ \Delta^{(1)}_1, \Delta^{(2)}_1\} = -\frac{1}{2} \Delta^{(1)}_1 \Delta^{(2)}_1$.
\end{example}

One can also consider the group $B_-$ of lower triangular matrices with
Poisson bracket
$$
\{ f^1 , f^2\}  = [r, f^1f^2 ] .
$$
The matrix elements of the inverse matrix $f^{-1}$ have Poisson brackets
of the same type (up to sign):
$$
\{ (f^{-1})^1 , (f^{-1})^2\}  = - [r, (f^{-1})^1(f^{-1})^2 ] .
$$
We shall denote by $\Lambda^{(k)}_l$ the solid minor of the matrix $f^{-1}$
formed by the columns with labels $n-k+1, \dots, n-k+l$ and the last
$l$ rows (see Figure \ref{lowtrmin}). 

\begin{figure}[!htbp]
\label{lowtrmin}
\begin{center}
\includegraphics[width=7cm]{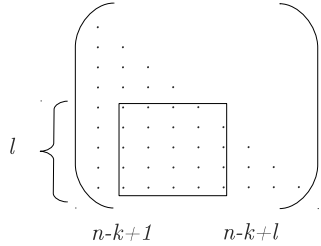}
\caption{minor $\Lambda^{(k)}_l$ }
\end{center}
\end{figure}

\no Similarly to Theorem \ref{thm:DeltaPoisson}, one can show that
$\Lambda^{(k)}_l$'s are log-canonical coordinates on an open dense chart
in $B_-$ and that their Poisson brackets are of the form
$$
\{ \Lambda^{(k)}_l, \Lambda^{(p)}_q \} = \frac{1}{2} \epsilon(k-p) ( R - C ) \Lambda^{(k)}_l \Lambda^{(p)}_q.
$$

\vskip 0.2cm

{\em Remark.} We could have chosen some solid minors of $f$ as log-canonical coordinates on $B_-$. 
Our choice  of solid minors of $f^{-1}$ is dictated by convenience of calculations in the next section.

\subsection{Poisson-Lie group $G^*_0$}
Denote by $B_-$ the group of lower triangular matrices.
For an element $g \in B_+$, let $g_d$ be its diagonal part ($(g_d)_{i,i}=g_{i,i}$ and
$(g_d)_{i,j} =0$ for $i \neq j$). The group $G^*_0$ is defined as
$$
G^*_0=\{ (g,f) \in B_+ \times B_-; \,\, g_d f_d=1\} 
$$
with product induced by the one of $B_+ \times B_-$. The standard Poisson bracket
on $G^*_0$ is defined by formulas
$$
\begin{array}{lll}
\{ g^1 , g^2\}  & = & [r, g^1g^2 ] ,  \\
\{ f^1 , f^2\}  & = & [r, f^1f^2 ] , \\
\{ g^1 , f^2\}  & = & [r_0, g^1f^2 ] .
\end{array}
$$

\begin{rem}
The corresponding Drinfeld double Lie group is $G \times G$, and the dual Poisson Lie group
to $G_0^*$ is $G_0=\{ (f,g) \in B_- \times B_+; \,\, g_d=f_d\}$.

\end{rem}

Consider the functions $\Delta^{(k)}_l$ and $\Lambda^{(k)}_l$ on $G^*_0$
for $n \geq k \geq l\geq 1$.  Note that the relation $g_df_d=1$ implies
the relation on the minors $\Delta^{(k)}_k= \Lambda^{(k)}_k$ for all $1\leq k \leq n$.

\begin{thm}
Functions $\Delta^{(k)}_l$ and $\Lambda^{(k)}_l$ (modulo the relations
$\Delta^{(k)}_k= \Lambda^{(k)}_k$) are log-canonical coordinates on $G^*_0$.
Their Poisson brackets are given by
$$
\begin{array}{lll}
\{ \Delta^{(k)}_l, \Delta^{(p)}_q \}  & = & \frac{1}{2} \varepsilon(k-p) ( C - R ) \Delta^{(k)}_l \Delta^{(p)}_q, \\
\{ \Lambda^{(k)}_l, \Lambda^{(p)}_q \} & = & \frac{1}{2} \varepsilon(k-p) ( R - C ) \Lambda^{(k)}_l 
\Lambda^{(p)}_q, \\
\{ \Delta^{(k)}_l, \Lambda^{(p)}_q \} & = & \frac{1}{2} (A - B)  \Delta^{(k)}_l \Lambda^{(p)}_q   ,
\end{array}
$$
where $A$ is the number of columns of the minor $\Delta^{(k)}_l$ which have the same labels 
as rows of the minor $\Lambda^{(p)}_q$, and B is the number of rows of the minor $\Delta^{(k)}_l$ 
which have the same labels  as columns of the minor $\Lambda^{(p)}_q$.
\end{thm}

\begin{proof}
Natural projections $G^*_0 \to B_+$ and $G^*_0 \to B_-$ given by formulas
$(g,f) \to g$ and $(g,f) \to f$  are Poisson maps. Hence, the Poisson brackets
$\{ \Delta^{(k)}_l, \Delta^{(p)}_q \}$ and $\{ \Lambda^{(k)}_l, \Lambda^{(p)}_q \}$
are given by Theorem \ref{thm:DeltaPoisson} and by the comment in the end
of the previous section.

For the brackets $\{ \Delta^{(k)}_l, \Lambda^{(p)}_q \}$ we have
$$
\{ g^1, (f^{-1})^2 \} = g^1 r_0 (f^{-1})^2 - (f^{-1})^2 r_0 g^1.
$$
We shall use the formula (\ref{prrez4}) for two arbitrary minors from Appendix \ref{appA} which reads:

$$
\{g_{IJ}, (f^{-1})_{ST} \} = \frac{1}{2} \, \big( | J \cap S | - |I \cap T|\big)  \, g_{IJ} (f^{-1})_{ST}.
$$
\no By definition, $A=|J \cap S| $ and $B= |I \cap T|$ and the expression for 
$\{ \Delta^{(k)}_l, \Lambda^{(p)}_q \}$ follows. 

\end{proof}

We will be interested in the real form of the group $G^*_0$ where one imposes a relation $f^*= g^{-1}$ on the components $(g,f)$ of the group element. Note that
on this real form we have $\Lambda^{(k)}_l = \overline{\Delta}^{(k)}_l$, and  the values of $\Delta^{(k)}_k= \Lambda^{(k)}_k$ are real.

\begin{thm}
The bracket $\{ \cdot, \cdot\}^{\mathbb{R}} = i \{ \cdot, \cdot\}$ is a real Poisson bracket on $G^*_0$,
and it  verifies the reality conditions \eqref{reality} in log-canonical  coordinates 
$\Delta^{(k)}_l \in \mathbb{C}$ for
$n \geq k > l \geq 1$ and $\Delta^{(k)}_k \in \mathbb{R}$ for $1 \leq k \leq n$.
\end{thm}

\begin{proof}

In order to check that the Poisson bracket $\{ \cdot, \cdot\}^{\mathbb{R}}$ is real, 
we compute

\begin{multline*}\{\overline{g}^1,\overline{g}^2\}^{\mathbb{R}}=\{(f^{-1})^{1t},(f^{-1})^{2t}\}^{\mathbb{R}}=i[r^{t},(f^{-1})^{1t}(f^{-1})^{2t}]=\\=-i[r,(f^{-1})^{1t}(f^{-1})^{2t}] =-i[r,\overline{g}^1\overline{g}^2 ]=\overline{\{g^1,g^2\}^{\mathbb{R}}} 
\end{multline*}

\no Here we have used that the element $r + r^t$, where $r^t= \frac{1}{2} \sum_{i=1}^n e_{i,i} \otimes e_{i,i} + \sum_{i<j} e_{j,i} \otimes e_{i,j}$, is invariant under the diagonal action of ${\rm GL}(n)$ by conjugation. Thus, one can replace $r^t$ by $-r$ in the commutator. The same calculation can be repeated for the bracket $\{f^1,f^2\}$. For the mixed bracket, we write

\begin{multline*}\{\overline{g}^1,\overline{f}^2\}^{\mathbb{R}}=\{(f^{-1})^{1t},(g^{-1})^{2t}\}^{\mathbb{R}}=-i[r_0^{t},(f^{-1})^{1t}(g^{-1})^{2t}]=\\=-i[r_0,(f^{-1})^{1t},(g^{-1})^{2t}] =-i[r_0,\overline{g}^1\overline{f}^2 ]=\overline{\{g^1,f^2\}^{\mathbb{R}}} 
\end{multline*}

 The bracket $\{ \cdot, \cdot\}^{\mathbb{R}}$ verifies the conditions \eqref{reality} since all its
 defining tensors are purely imaginary. 
 \end{proof} 
 
In the next section, we denote the Poisson structure $\{ \cdot, \cdot\}^{\mathbb{R}}$ on $G^{*}_0$ by $\pi_{G^{*}_0}$.

 \begin{example}
 For $n=2$, we have three coordinate functions $\Delta^{(1)}_1, \Delta^{(2)}_2 \in \R,
 \Delta^{(2)}_1 \in \CC$.  The non-vanishing Poisson brackets read
 $$
\{ \Delta^{(1)}_1, \Delta^{(2)}_1\}^\R = -\frac{i}{2} \, \Delta^{(1)}_1 \Delta^{(2)}_1, \hskip 0.3cm
\{ \Delta^{(1)}_1, \bar{\Delta}^{(2)}_1\}^\R =  \frac{i}{2} \, \Delta^{(1)}_1 \bar{\Delta}^{(2)}_1.
$$ 
 We can actually put the Casimir function $\Delta^{(2)}_2$ equal to one  and consider upper- and lower-triangular matrices with unit determinant.
 \end{example}

\section{Poisson bracket on Poisson-Lie group $G^*$}
\label{PoisG}

The definition of the Poisson-Lie group $G^*$ is due to Semenov-Tian-Shansky \cite{STS}
and Lu-Weinstein \cite{LW}, 
$$
G^* = \{ (g,f) \in B_+ \times B_-; \,\, g_df_d=1\} .
$$
As groups, $G^*$ and $G_0^*$ are isomorphic.
However, their Poisson structures are different:
$$
\begin{array}{lll}
\{ g^1 , g^2\}  & = & [r, g^1g^2 ] ,  \\
\{ f^1 , f^2\}  & = & [r, f^1f^2 ] , \\
\{ g^1 , f^2\}  & = & [r, g^1f^2 ] .
\end{array}
$$

\begin{rem}
The corresponding Drinfeld double is again (as in the case of $G_0^*$) $G\times G$, and the dual 
Poisson-Lie group is a copy of $G \cong \{ (g,g) \in G \times G; \,\, g \in G\}$.

\end{rem}

Note that the only difference with respect to the Poisson bracket on $G_0^*$ is in the brackets between $g$ and $f$, whereas the brackets between $g$'s and the brackets between $f$'s are exactly the same as for $G_0^*$. In view of this remark, the following statement is obvious:

\begin{prop}
For the Poisson bracket on $G^*$, we have
$$
\begin{array}{lll}
\{ \Delta^{(k)}_l, \Delta^{(p)}_q \}  & = & \frac{1}{2} \varepsilon(k-p) ( C - R ) \Delta^{(k)}_l \Delta^{(p)}_q, \\
\{ \Lambda^{(k)}_l, \Lambda^{(p)}_q \} & = & \frac{1}{2} \varepsilon(k-p) ( R - C ) \Lambda^{(k)}_l 
\Lambda^{(p)}_q .
\end{array}
$$
\end{prop}

Note that the Poisson-Lie group $G^*$ also admits a real form defined by the equation 
$f^*=g^{-1}$. As before, this implies $\Lambda^{(k)}_l= \bar{\Delta}^{(k)}_l$. In contrast
to the group $G^*_0$, the Poisson brackets between $\Delta$'s and $\bar{\Delta}$'s are no
longer log-canonical. More precisely, we can use equation \eqref{grf-frg} (see Appendix~\ref{appA}) for arbitrary minors of matrices $g$ and $f^{-1}$ which reads
\begin{multline*}
\{ g_{IJ}, (f^{-1})_{ST} \}
=\sum_{u<v}\chi_J(v) \chi_S(v) \, g_{I,\sigma_{v,u}(J)} (f^{-1})_{\sigma_{v,u}(S), T}-\\
-\sum_{u<v}\chi_I(u) \chi_T(u) \, g_{\sigma_{u,v}(I),J} (f^{-1})_{S,\sigma_{u,v}(T)}
 + \frac{1}{2} \, (| J \cap S| -| I \cap T |)\, g_{IJ}  (f^{-1})_{ST} 
\end{multline*}

Recall \cite{FZ2} that all minors of the matrix $g$ are Laurent polynomials in the 
minors $\Delta^{(k)}_l$ (see Appendix \ref{appB}), and similarly all minors of the matrix $f^{-1}$ are Laurent polynomials in the minors $\Lambda^{(k)}_l=\bar{\Delta}^{(k)}_l$. Hence, the right hand side of the formula above is a Laurent polynomial in the minors $\Delta^{(k)}_l, \bar{\Delta}^{(k)}_l$ and one can apply the tropicalization machinery of Section \ref{logtrop}.

\begin{example}   \label{ex:Gstarn=2}
For $n=2$, we use the same functions as in the case of $G^*_0$,
$\Delta^{(1)}_1, \Delta^{(2)}_2 \in \R, \Delta^{(2)}_1 \in \CC$. The Poisson brackets 
 $$
\{ \Delta^{(1)}_1, \Delta^{(2)}_1\}^\mathbb{R} = -\frac{i}{2} \, \Delta^{(1)}_1 \Delta^{(2)}_1, \hskip 0.3cm
\{ \Delta^{(1)}_1, \bar{\Delta}^{(2)}_1\}^\mathbb{R} =  \frac{i}{2} \, \Delta^{(1)}_1 \bar{\Delta}^{(2)}_1
$$
are the same as for $\pi_{G^*_0}$. The new contribution is 
$$
\{ \Delta^{(2)}_1, \bar{\Delta}^{(2)}_1\}^\mathbb{R} =  i \big( g_{11}(f^{-1})_{11} - g_{22}(f^{-1})_{22} \big)=
i \left( \frac{\Delta^{(2)}_2}{\Delta^{(1)}_1} \right)^2 - i \left( \Delta^{(1)}_1\right)^2.
$$ 
Here we have used that $\Delta^{(2)}_2 = g_{11} g_{22}$. The minor $\Delta^{(2)}_2$ is a Casimir
function.

Using notation $\Delta^{(k)}_k = \exp(t \zeta^{(k)}_k)$ (for convenience we are using
the notation $\zeta^{(k)}_k$ instead of $\xi^{(k)}_k$) and  $\Delta^{(2)}_1=\exp(t \zeta^{(2)}_1 + i \phi^{(2)}_1)$  we obtain the following inequalities defining the cone $\C(\pi, \Delta, \bar{\Delta})$:
$$
\zeta^{(2)}_1 > \zeta^{(2)}_2 - \zeta^{(1)}_1, \hskip 0.3cm
\zeta^{(2)}_1 > \zeta^{(1)}_1.
$$
\no The non-vanishing component of the Poisson bracket $\{ \cdot, \cdot \}_\infty$ reads
$$
\{\zeta^{(1)}_1, \phi^{(2)}_1\}_\infty = -\frac{1}{2} .
$$
\no Both $\zeta^{(2)}_2$ and $\zeta^{(2)}_1$ are Casimir functions for this bracket.

\end{example}

\begin{prop}
In coordinates $\Delta^{(k)}_l, \bar{\Delta}^{(k)}_l$, the log-canonical part of the Poisson bracket $\pi_{G^*}$ is equal to the Poisson bracket $\pi_{G^*_0}$.
\end{prop}

Recall that in coordinates $\Delta^{(k)}_l, \bar{\Delta}^{(k)}_l$ the Poisson bracket $\pi_{G^*_0}$ verifies reality conditions \eqref{reality}. Hence, so does the Poisson bracket  $\pi_{G^*}$. 

\begin{proof}
Let $a: \R_+ \times B_+ \to B_+$ be the following action of the multiplicative group
$\R_+$:
$$
a: (\lambda, g) \to d_\lambda g d_\lambda^{-1}
$$
with  $d_\lambda={\rm diag}(\lambda, \lambda^2, \dots, \lambda^n)$. This action
introduces a grading on the set of regular functions on $B_+$. In particular,
the grading of the minors is given by
\be \label{grading}
{\rm deg}(g_{IJ}) = \sum_{s=1}^l i_s - \sum_{r=1}^l j_r.
\ee
Let $g^{(k)}$ be the submatrix of $g$ with rows and columns $\{ n-k+1, \dots, n\}$ (the lower right corner of size $k$). Note that the minor $\Delta^{(k)}_l$ is the
minor of $g^{(k)}$ of size $l$ which has the lowest possible grading. 

For the matrix $f^{-1}=g^*$ the action of $\mathbb{R}_+$ reads
$(\lambda, f^{-1}) \to d_\lambda^{-1} f^{-1} d_\lambda$, and the grading is given by
$$
{\rm deg}(f^{-1}_{IJ}) = \sum_{r=1}^l j_r - \sum_{s=1}^l i_s.
$$
In particular, the minor $\Lambda^{(k)}_l=\bar{\Delta}^{(k)}_l$ is the minor of
$(g^{(k)})^*$ with the lowest possible grading.

Consider the Poisson bracket $\{ \Delta^{(k)}_l, \bar{\Delta}^{(p)}_q\}$.
Note that the minors $g_{I,\sigma_{v,u}(J)}$ and $g_{\sigma_{u,v}(I),J}$
for $u<v$ are in fact minors of the matrix $g^{(k)}$. Indeed, in $g_{\sigma_{u,v}(I),J}$ we are replacing the row number $u$ with the row number $v>u$, hence we cannot leave the range $\{ n-k+1, \dots, n\}$. In $g_{I,\sigma_{v,u}(J)}$, we are replacing the column number $v$ with the column number $u<v$. However, $g$ is an upper triangular matrix, and its minor $g_{IJ}$ with $I \subset \{ n-k+1, \dots, n\}$  is non-vanishing only if $J \subset \{ n-k+1, \dots, n\}$. A similar consideration applies to the minors of $f^{-1}$. 

We conclude that the first two terms on the right hand side of the Poisson bracket $\{ \Delta^{(k)}_l, \bar{\Delta}^{(p)}_q\}$ are linear combinations of functions of degree strictly greater than the one of the product $\Delta^{(k)}_l \bar{\Delta}^{(p)}_q$. Hence, the only log-canonical contribution in this Poisson bracket comes from the third term which coincides with the Poisson bracket on $G^*_0$,
$$
\{ \Delta^{(k)}_l, \bar{\Delta}^{(p)}_q \}^{\mathbb{R}}_{G^{*}_0}
= \frac{i}{2} \, (A-B)\, \Delta^{(k)}_l \bar{\Delta}^{(p)}_q ,
$$
as required.
\end{proof}

One of the main results of this paper is the description of the tropicalization of the Poisson bracket $\pi_{G^*}$. We shall use the following notation for the scaling limit:
$\Delta^{(k)}_l = \exp(t \zeta^{(k)}_l + i \phi^{(k)}_l)$. Note that $\phi^{(k)}_k=0$ for all $k=1, \dots, n$ and we are using the more uniform notation $\zeta^{(k)}_k$ for the scaling limit of the real variables $\Delta^{(k)}_k$ instead of (the more logical) $\xi^{(k)}_k$.

\begin{thm}
The cone $\C(\pi_{G^*}; \Delta, \bar{\Delta})$ is isomorphic to the Gelfand-Zeitlin cone $\C_{GZ}$. The isomorphism $\sigma: \C_{GZ} \to \C(\pi_{G^*}; \Delta, \bar{\Delta})$ is given by formula
$$
\zeta^{(k)}_l = \lambda^{(k)}_1 + \dots + \lambda^{(k)}_l.
$$
\end{thm}

In the proof, we are using the machinery of planar networks and the notion of the tropical Gelfand-Zeitlin map. For more information and notation, we refer the reader to Appendices \ref{appB} and \ref{appC}.

\begin{proof}
The map $\sigma$ defines an isomorphism of vector spaces of dimensions $n(n+1)/2$. We shall first prove that $\sigma(\C_{GZ}) \subset \C(\pi_{G^*}; \Delta, \bar{\Delta})$. Recall that by Theorem 3 in \cite{aps}  the tropical Gelfand-Zeitlin map establishes a bijection between the Gelfand-Zeitlin cone and the principal chamber $\C_0 \subset \R^{n(n+1)/2}$.  On this  chamber, the weight of the multi-path $\gamma^{(k)}_l$ is strictly bigger than the weights of all the other $l$-paths in the subnetwork $\Gamma_s^{(k)}$.

We shall use the coordinates on $B_+$ defined by the planar network $\Gamma^{(n)}_s$ with the weights parametrized as $w(e)=\exp(t\zeta(e) + i \phi(e))$. Then, the weight  of the multi-path $\gamma$ is given by the function
$$
h_\gamma(\zeta, \phi)= \prod_e w(e) = \exp\big(t \sum_{e \in \gamma} \zeta(e) 
+ i \sum_{e\in \gamma} \phi(e) \big) .
$$
By Lindstr\"om Lemma \cite{FZ}, minors of the matrix $M(\Gamma, w)$ are linear combinations of functions $h_\gamma(\zeta, \phi)$. Hence, we obtain the following expression for the Poisson bracket of two minors $\Delta^{(k)}_l$ and $\bar{\Delta}^{(x)}_y$:
$$
\{ \Delta^{(k)}_l, \bar{\Delta}^{(x)}_y \}^{\mathbb{R} }
= \frac{i}{2} \, (A-B)\, \Delta^{(k)}_l \bar{\Delta}^{(x)}_y +
\sum_{\gamma, \tilde{\gamma}} c(\gamma, \tilde{\gamma}) h_\gamma(\zeta, \phi) 
h_{\tilde{\gamma}}(\zeta, - \phi).
$$
\no Here $\gamma$'s are paths in $\Gamma_s^{(k)}$, $\tilde{\gamma}$'s are paths in $\Gamma_s^{(x)}$, the complex conjugation corresponds to replacing $\phi(e) \mapsto - \phi(e)$ and $c(\gamma,\tilde{\gamma})$ are some coefficients.

Note that 
$$|h_\gamma(\zeta, \phi)|=\exp\big(t\sum_{e\in\gamma}\zeta(e)\big) \indent \text{and} \indent |\Delta^{(k)}_{l}|=\exp\big(t\sum_{e\in\gamma^{k}_l}\zeta(e)\big),$$

\no and assume that parameters $\zeta$ belong to the interior of the principal chamber $\C_0$. Then, the maximality property of the paths $\gamma^{(k)}_l$ implies
$$
\sum_{e \in \gamma^{(k)}_l} \zeta(e) >  \sum_{e \in \gamma} \zeta(e), 
\hskip 0.3cm
\sum_{e \in \gamma^{(x)}_y} \zeta(e) >  \sum_{e \in \tilde{\gamma}} \zeta(e)
$$
for all paths $\gamma, \tilde{\gamma}$ in the sum above. Hence, $\Delta^{(k)}_l$ dominates $h_\gamma(\zeta, \phi)$ and  $\bar{\Delta}^{(x)}_y$ dominates $h_{\tilde{\gamma}}(\zeta, -\phi).$ By definition of $\C(\pi_{G^*}; \Delta, \bar{\Delta})$, we conclude that $\zeta \in \C(\pi_{G^*}; \Delta, \bar{\Delta})$, as required.

Next, we shall show that $\C(\pi_{G^*}; \Delta, \bar{\Delta}) \subset \sigma(\C_{GZ})$.
In order to do that, we consider the Poisson bracket $\{ \Delta^{(k)}_l, \bar{\Delta}^{(k)}_l\}$.
By formula \eqref{grading} the weight of minors $\Delta^{(k)}_l$ and $\bar{\Delta}^{(k)}_l$ is given by
$$
\frac{\left( (n-k+1) + (n-k+l) \right) l}{2} - \frac{\left( ( n-l+1) + n \right) l}{2} = -l(k-l).
$$
The log-canonical contribution (of weight $-2l(k-l)$) vanishes since in this case $A=B$.
There are two contributions in the Poisson bracket of weight $-2l(k-l)+ 2$ which are of the form
$$
\begin{array}{lll}
\{ \Delta^{(k)}_l, \bar{\Delta}^{(k)}_l\}^{\mathbb{R}} & = & i\left| g_{\{n-k+1, \dots, n-k+l; \, n-l, n-l+2, \dots, n\}} \right|^2
\\
& - & i\left|g_{\{n-k+1, \dots, n-k+l-1, n-k+l+1; \, n-l+1, \dots, n\}} \right|^2   \vspace{0.2cm}\\
& + & 
 terms \,\, of  \,\, higher \,\, weight\,\, 
\end{array}
$$
Note that the explicit form of the minors $\Delta^{(k)}_l$ is as follows: $\Delta^{(k)}_l= \Delta_{\{n-k+1, \dots, n-k+l; \, n-l+1, \dots, n\}}$ (for Lindstr\"om Lemma of $\Delta^{(k)}_l$ see Figure~\ref{paths3_2}). By the Linstr\"om's Lemma, the minors $g_{\{n-k+1, \dots, n-k+l-1, n-k+l+1; \, n-l+1, \dots, n\}}$ and $g_{\{n-k+1, \dots, n-k+l; \, n-l, n-l+2, \dots, n\}}$ can be expressed as sums of weights of $l$-paths. Each product of two weights of $l$-paths in the expression for $|g |^2$ gives rise to a defining inequality for the cone $\C(\pi_{G^*}; \Delta, \bar{\Delta})$. Our task is to find the Gelfand-Zeitlin inequalities among them.

\begin{figure}[!htbp]  
\begin{center}
\includegraphics[width=8cm]{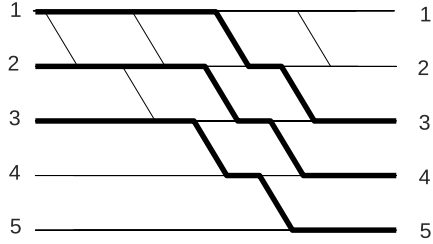}
\caption{Lindstr\"om Lemma presentation of $\Delta^{(k)}_l, k=5, l=3$ }
\end{center}
\label{paths3_2}
\end{figure}

We start with the minor $g_{\{n-k+1, \dots, n-k+l-1, n-k+l+1; \, n-l+1, \dots, n\}}$. In this case, choose in both $\Delta$ and $\bar{\Delta}$ the $l$-path $L_1$ shown on Fig. \ref{area1}  
(in fact, this is the lowest  possible $l$-path given by the Lindstr\"om's Lemma for the new minor). The picture shows that the ratio of $\Delta^{(k)}_l$ and of the weight of $L_1$  is given by
$$
\exp( t \times {\rm sum \,\, of \,\, weights \,\, of \,\, shadow \,\, regions}).
$$
By Lemma 9 in \cite{aps}, this sum of weights is given by
$$
u^{(k)}_l= \zeta^{(k)}_l + \zeta^{(k-1)}_{l-1} - \zeta^{(k)}_{l-1} - \zeta^{(k-1)}_l.
$$
The corresponding inequality reads $u^{(k)}_l >0$, and this gives one of the 
families of Gelfand-Zeitlin inequalities.

\begin{figure}[!htbp]  
\begin{center}
\includegraphics[width=8cm]{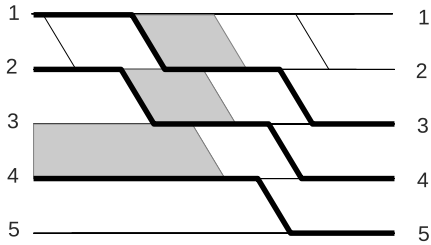}
\caption{Lowest $l$-path path $L_1$ and pictorial presentation of $\Delta^{(k)}_l/L_1$ }
\label{area1}
\end{center}
\end{figure}

In a similar fashion, we consider the minor $g_{\{n-k+1, \dots, n-k+l; \, n-l, n-l+2, \dots, n\}}$. In this case, we choose the highest $l$-path $L_2$ given by the Lindstr\"om Lemma. This
$l$-path is  shown on Fig. \ref{area2}. Again, we obtain a pictorial expression of the ratio of 
$\Delta^{(k)}_l$ and of the weight of $L_2$,
$$
\exp( - t \times {\rm sum \,\, of \,\, weights \,\, of \,\, shadow \,\, regions}).
$$
By Lemma 9 in \cite{aps}, the sum of weights reads
$$
v^{(k)}_l = \zeta^{(k-1)}_{l-1} + \zeta^{(k)}_{l+1} - \zeta^{(k)}_l - \zeta^{(k-1)}_l,
$$
and the corresponding inequality $v^{(k)}_l < 0$ gives the second family of Gelfand-Zeitlin
inequalities, as required. 

\begin{figure}[!htbp]  
\begin{center}
\includegraphics[width=8cm]{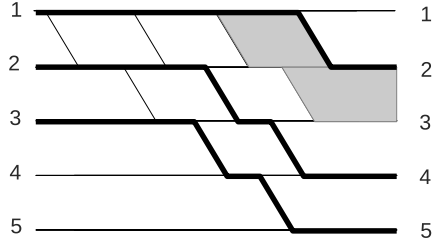}
\caption{Highest $l$-path path $L_2$ and pictorial presentation of $\Delta^{(k)}_l/L_2$ }
\label{area2}
\end{center}
\end{figure}

\end{proof}

\begin{example}
Making the substitution
$$
\xi^{(1)}_1 = \lambda^{(1)}_1, \hskip 0.3cm
\zeta^{(2)}_1 = \lambda^{(2)}_1, \hskip 0.3cm
\xi^{(2)}_2 = \lambda^{(2)}_1 + \lambda^{(2)}_2
$$
one can easily check that the inequalities of Example \ref{ex:Gstarn=2}
are equivalent to the Gelfand-Zeitlin inequalities for $n=2$:
$$
\lambda^{(2)}_1 > \lambda^{(1)}_1 > \lambda^{(2)}_2.
$$
\end{example}

\begin{thm} \label{thmcan}
For $\pi_{G^*}$ in coordinates $\Delta, \bar{\Delta}$, the Poisson bracket
$\{ \cdot, \cdot\}_\infty$ has the following properties:
$$
\{ \zeta^{(k)}_l, \phi^{(p)}_q\}_\infty = 0
$$
if $k\geq p$, or if $k<p$ and $k-l \geq p-q$. Furthermore,
$$
\{ \zeta^{(k)}_l, \phi^{(k+1)}_l \}_\infty = -\frac{1}{2}.
$$ 
\end{thm}

\begin{proof}
First, we combine the formulas

$$
\{ \Delta^{(k)}_l, \Delta^{(p)}_q \}^\mathbb{R}_{G_0^{*}} = \frac{i}{2} \, \ep(k-p) \, (C-R) \, \Delta^{(k)}_l \Delta^{(p)}_q,
$$
$$
\{ \Delta^{(k)}_l, \bar{\Delta}^{(p)}_q \}^\mathbb{R}_{G_0^{*}} = \frac{i}{2} \, (A-B) \, \Delta^{(k)}_l \bar{\Delta}^{(p)}_q
$$
with one of the reality conditions (see equations \eqref{brs}) to obtain
$$
\{ \zeta^{(k)}_l, \phi^{(p)}_q\}_\infty = \frac{1}{4} \, \left( \ep(k-p)  (C-R) - (A-B) \right).
$$
Note that $A$ and $B$ for the pair of minors $\Delta^{(k)}_l$ and $ \bar{\Delta}^{(p)}_q $
coincide with $C$ and $R$ for the pair of minors $\Delta^{(k)}_l$ and $ \Delta^{(p)}_q $.
Hence, the expression for the Poisson bracket simplifies as follows
$$
\{ \zeta^{(k)}_l, \phi^{(p)}_q\}_\infty = \frac{1}{4} \, ( \ep(k-p) -1) \, (C-R).
$$
Now it is obvious that the bracket vanishes for $k>p$ since in this case $\ep(k-p)=1$.
If $k \leq p$ and $k-l \geq p-q$, the submatrix corresponding to $\Delta^{(k)}_l$ is contained
in the submatrix corresponding to $\Delta^{(p)}_q$. Then, $C=R=l$ and the
Poisson bracket vanishes, as required.

Finally, for $p=k+1$ and $q=l$ we have $C=l$ while $R=l-1$ which yields
$$
\{ \zeta^{(k)}_l, \phi^{(k+1)}_l \}_\infty = -\frac{1}{2}.
$$
\end{proof}

The  following propositions are easy consequences of Theorem \ref{thmcan}.

\begin{prop}
Functions $\zeta^{(n)}_l$ for $l=1, \dots, n$ are Casimir functions for the bracket
$\{ \cdot, \cdot\}_\infty$.
\end{prop}

\begin{proof}
This statement is obvious since for $k=n$ the condition $k \geq p$ is always verified,
and we have
$$
\{ \zeta^{(n)}_l, \phi^{(p)}_q\}_\infty = 0 
$$
for all values of $p$ and $q$.
\end{proof}

\begin{prop}
Symplectic leaves of the Poisson bracket $\{ \cdot, \cdot\}_\infty$ are 
hyperplanes of constant $\zeta^{(n)}_l$ for $l=1, \dots, n$. The Liouville
form on symplectic leaves is given by:
$$
\mathcal{L} = \prod_{k=1}^{n-1} \prod_{l=1}^k \,\, \left( 2 d \phi^{(k+1)}_l\, \wedge d \zeta^{(k)}_l \right) .
$$
\end{prop}

\begin{proof}
Note that the number of variables $\zeta^{(k)}_l$ with $k=1, \dots, n-1$ and $l=1, \dots, k$
is exactly equal to the number of variables $\phi^{(p)}_q$ with $p=2, \dots, n$ and $q=1, \dots, p-1$.
Let us order the variables $\zeta$ and $\phi$ in such a way that the variables with higher $k$ come first,
and among variables with equal $k$ the ones with smaller $l$ come first. For instance, for $n=3$
we get the following order on $\zeta$'s: $\zeta^{(2)}_1, \zeta^{(2)}_2, \zeta^{(1)}_1$, and the order
on $\phi$'s: $\phi^{(3)}_1, \phi^{(3)}_2, \phi^{(2)}_1$. Now to every $\zeta$ (and to every $\phi$) 
we can associate its number in the order of $\zeta$'s (respectively, in the order of $\phi$'s).

With this order, the Poisson bracket
$\{ \zeta^{(k)}_l, \phi^{(p)}_q\}_\infty$  is given by a lower triangular form since 
$$
\{ \zeta_M, \phi_N\}_\infty =0 
$$
if $M < N$. The diagonal entries are non-vanishing and equal to
$$
\{ \zeta_N, \phi_N\}_\infty = - \frac{1}{2}.
$$
Hence, the tangent vectors $\partial/\partial \zeta^{(k)}_l$
and $\partial/\partial \phi^{(k+1)}_l$ for $k=1, \dots, n-1$ and $l=1, \dots, k$ span the sympletic leaf. The matrix of the symplectic form is the inverse of the transposed matrix of Poisson brackets. It is also lower triangular with $(-2)$'s as diagonal entries which implies the formula for the Liouville form.
\end{proof}

\begin{thm}
The Poisson manifold $\C(\pi_{G^*}; \Delta, \bar{\Delta}) \times \mathbb{T}^{n(n-1)/2}$ equipped
with the Poisson bracket $\{ \cdot, \cdot\}_\infty$ is isomorphic to the Gelfand-Zeitlin completely integrable system.
\end{thm}

\begin{proof}
Recall that the Gelfand-Zeitlin integrable system is described by the Poisson manifold
$\C_{GZ} \times \mathbb{T}^{n(n-1)/2}$. Coordinates on $\C_{GZ}$, $\lambda^{(k)}_l,
k=1, \dots, n, l=1, \dots, k$ satisfy the interlacing inequalities and can be interpreted as action variables of the integrable system. Coordinates $\psi^{(k)}_l, k=1, \dots, n-1, l=1, \dots, k$ on $\mathbb{T}^{n(n-1)/2}$ become angle variables. The Poisson bracket is given by
$$
\{ \lambda^{(k)}_l, \psi^{(x)}_y\} = - \delta_{k,x} \delta_{l,y}.
$$

We claim that there is a unique Poisson isomorphism 
$$
\C_{GZ} \times \mathbb{T}^{n(n-1)/2} \to \C(\pi_{G^*}; \Delta, \bar{\Delta}) \times \mathbb{T}^{n(n-1)/2}
$$
such that
\be \label{eq:lambda_zeta}
\zeta^{(k)}_l = \frac{1}{2} \left( \lambda^{(k)}_1 + \dots + \lambda^{(k)}_l \right).
\ee
for $k=1, \dots, n$ and  $l=1, \dots, k$, and 
\be \label{eq:phi_psi}
\phi^{(k)}_l = \psi^{(k-1)}_l + { linear \,\, combination \,\, of \,\, } \psi'{ s\,\, higher \,\, in \,\, the \,\, order}
\ee
for $k=2, \dots, n$ and $ l=1, \dots, k-1$.

Indeed, the map \eqref{eq:lambda_zeta} is a bijection between $\C_{GZ}$ and $\C(\pi_{G^*}; \Delta, \bar{\Delta})$, and the map \eqref{eq:phi_psi} is an automorphism of the torus $\mathbb{T}^{n(n-1)/2}$.
For the Poisson brackets, we have
$$
\{ \zeta_M, \phi_N \} = 0
$$
for $M<N$ since $\{ \lambda_a, \psi_b\} =0$ for $a<b$ and $\zeta_M$ is a linear
combination of $\lambda_a$'s with $a \leq M$ and $\phi_N$ is a linear combination of 
$\psi$'s with $b \geq N$. 

Next, we obtain
$$
\{ \zeta_N, \phi_N\} =- \frac{1}{2}
$$
because $\zeta_N$ is a sum of $\lambda_N/2$ and a linear combination of $\lambda_a$'s
with $a < N$, and $\psi_N$ is a sum of $\psi_N$ and a linear combination of $\psi_b$'s 
with $b >N$.

Finally, there is a unique choice of linear combinations in \eqref{eq:phi_psi} such that
the constants $\{ \zeta_M, \phi_N\}$ for $M>N$ are set to the values given by $\{ \cdot, \cdot\}_\infty$.

\end{proof}

\newpage
\appendix 
\section{ $r$-matrix Poisson brackets} \label{appA}

Let $[n]=\{ 1, \dots, n\}$, $I= \{ i_1, i_2, \dots, i_c\} \subset [n]$ and $J= \{ j_1, j_2, \dots, j_c\} \subset [n]$. We shall use the following notation:
\begin{itemize}
 \item  $M_{IJ}$ denotes the minor of a matrix $M\in {\rm Mat}_n$ with rows labeled by elements
 of $I$ and columns labeled by elements of $J$;
 \item $\chi_I$ is the characteristic function of $I$, so that $\chi_I(k) = 1$ if $k \in I$
and $\chi_I(k)=0$ otherwise;
\item  for $k\in I$,  $\sigma_{k,l}(I)$ is the set  obtained from $I$ after replacing $k$ by $l$.

\end{itemize}

\begin{prop} \label{propbree}
Let $i, j \in [n]$ and let 
\be  
\{ L^1, M^2\} = r'\,L^1 M^2
\ee
be a skew-symmetric bracket on ${\rm Mat}_n \times {\rm Mat}_n$. Then
\be \label{bracketeer'}
\{ L_{IJ}, M_{ST} \} =\sum_{u<v} \chi_I(u) \chi_S(v) \, L_{\sigma_{u,v}(I), J} M_{\sigma_{v,u}(S), T}
\ee
\end{prop}

\begin{rem}
Note that if $v \in I$ the minor $L_{\sigma_{u,v}(I), J}$ vanishes since it contains
two identical rows. The same applies to the case of $u \in S$.
\end{rem}

\begin{proof}

Let us first consider a bracket 

\be  \label{bracketee}
\{ L^1, M^2\} = (e_{uv} \otimes e_{vu}) L^1 M^2
\ee

\no and prove that for such a bracket

\be \label{bracketeerez}
\{ L_{IJ}, M_{ST} \} =\chi_I(u) \chi_S(v) \, L_{\sigma_{u,v}(I), J} M_{\sigma_{v,u}(S), T}.
\ee

\no Taking matrix elements $(i,j)$ in the first space and $(s,t)$ in the second space in the
formula \eqref{bracketee} we get
\be \label{bracketone}
\{ L_{ij}, M_{st} \} = \delta_{iu} \delta_{sv} L_{vj} M_{ut}
= \chi_{ \{ i\} }(u)  \chi_{ \{ s\} }(v) \, L_{\sigma_{u,v}( i), j} M_{\sigma_{v,u}(s), t}.
\ee

\no This is exactly the equation \eqref{bracketeerez} where the
sets $I,J,S,T$ consist of one element each.
The minors are linear in their rows and the Poisson bracket is a
derivation on each factor. Hence, we obtain equation \eqref{bracketeerez}
in the general case
by applying equation \eqref{bracketone} to each pair or rows
of matrices $M$ and $L$ and summing up the results.

Since $r'= \sum_{u<v} e_{uv} \otimes e_{vu} $, equation 
\eqref{bracketeer'} directly follows from  equation (\ref{bracketeerez}) by  taking the sum:

$$
\{ L^1, M^2 \} =\sum_{u<v} e_{uv} \otimes e_{vu} \, L^1 M^2=\sum_{u<v}\chi_I(u) \chi_S(v) \, L_{\sigma_{u,v}(I), J} M_{\sigma_{v,u}(S), T}
$$

\end{proof}

\no A similar argument shows that for the bracket $\{ L^1, M^2\} = L^1 M^2 r'$ one obtains
\be \label{prrez1}
\{ L_{IJ}, M_{ST} \} = \sum_{u<v}\chi_J(v) \chi_T(u)L_{I,\sigma_{v,u}(J)} M_{S,\sigma_{u,v}(T)},
\ee
\no And for the bracket $\{ L^1, M^2\} = L^1 r' M^2$ one gets
\be \label{prrez2}
\{ L_{IJ}, M_{ST} \} = \sum_{u<v}\chi_J(v) \chi_S(v) \, L_{I,\sigma_{v,u}(J)} M_{\sigma_{v,u}(S),T}
\ee

\begin{prop} \label{propbree2}
For the skew-symmetric bracket 
\be \label{bracketr0}
\{ L^1, M^2 \} = r_0 L^1 M^2
\ee
on ${\rm Mat}_n \times {\rm Mat}_n$, we have
\be \label{minorsr0}
\{ L_{IJ}, M_{ST} \} = \frac{1}{2} \, | I \cap S| \, L_{IJ}  M_{ST}.
\ee
\end{prop}


\begin{proof}

Recall that $r_0= \frac{1}{2} \, \sum_k e_{kk} \otimes e_{kk} $. The proof is similar to that of Proposition \ref{propbree}. We compute,
$$
\{ L_{IJ}, M_{ST} \} = \frac{1}{2} \, \sum_k \, \chi_I(k) \chi_S(k) L_{IJ} M_{ST}
=\frac{1}{2} \, | I \cap S| \, L_{IJ}  M_{ST}.
$$

\end{proof}

\no Similarly, for the bracket $\{ L^1, M^2\} = L^1 M^2 r_0$ we have
\be \label{prrez3}
\{L_{IJ}, M_{ST} \} = \frac{1}{2} \, | J \cap T | \, L_{IJ} M_{ST},
\ee
\no and for the bracket $\{ L^1, M^2\} = L^1 r_0 M^2$
\be \label{prrez4}
\{L_{IJ}, M_{ST} \} = \frac{1}{2} \, | J \cap S | \, L_{IJ} M_{ST}.
\ee

\begin{thm}
 Let $i, j \in [n]$ and let 
\be 
\{ L^1, M^2\} = [r\,, L^1 M^2]
\ee
be a skew-symmetric bracket on ${\rm Mat}_n \times {\rm Mat}_n$. Then,

\begin{multline}    \label{bracketdetailed}
\{ L_{IJ}, M_{ST} \}
=\sum_{u<v}\chi_I(u) \chi_S(v) \, L_{\sigma_{u,v}(I), J} M_{\sigma_{v,u}(S), T}-\\
-\sum_{u<v}\chi_J(v) \chi_T(u) \, L_{I,\sigma_{v,u}(J)} M_{S,\sigma_{u,v}(T)}
 + \frac{1}{2} \, (| I \cap S| -| J \cap T |)\, L_{IJ}  M_{ST} 
\end{multline}

\end{thm}

\begin{proof}
 The theorem directly follows from Propositions \ref{propbree} and \ref{propbree2} and equations
 (\ref{prrez1}) and (\ref{prrez3}) .
\end{proof}

\begin{thm}
 Let $i, j \in [n]$ and let 
\be  
\{ L^1, M^2\} = L^1\,r\,M^2-M^2\,r\,L^1
\ee
be a skew-symmetric bracket on ${\rm Mat}_n \times {\rm Mat}_n$. Then,

\begin{multline}    \label{grf-frg}
\{ L_{IJ}, M_{ST} \}
=\sum_{u<v}\chi_J(v) \chi_S(v) \, L_{I,\sigma_{v,u}(J)} M_{\sigma_{v,u}(S), T}-\\
-\sum_{u<v}\chi_I(u) \chi_T(u) \, L_{\sigma_{u,v}(I),J} M_{S,\sigma_{u,v}(T)}
 + \frac{1}{2} \, (| J \cap S| -| I \cap T |)\, L_{IJ}  M_{ST} 
\end{multline}

\end{thm}

\begin{proof}
 The theorem follows from equations (\ref{prrez2}) and (\ref{prrez4}) .
\end{proof}

\section{Planar networks}  \label{appB}

A planar network $\Gamma$ of type $n$ is a finite planar oriented graph which satisfies
the following conditions:

\begin{itemize}

\item
It is contained between two straight vertical lines $L$ and $R$.

\item
Its edges are segments of straight lines, and their horizontal projections are non-vanishing. All the edges are oriented in such a way that their horizontal
projections are positive.

\item
It has exactly $n$ sources on $L$ and exactly $n$ sinks on $R$, the number $n$ is called the type of a planar network.
\end{itemize}

\no Let $V\Gamma$ be the set of vertices of $\Gamma$ and $E\Gamma$ - the set of edges. A weighting of a planar network is a map $w: E\Gamma \to \CC$. One can associate a matrix to a planar network $\Gamma$ with weighting $w$ in a following way:
$$
M(\Gamma, w)_{ij} = \sum_{\gamma \in P\Gamma_{ij}} \, \prod_{e \in \gamma} \, w(e),
$$
where $P\Gamma_{ij}$ is the set of paths in $\Gamma$ starting in the source with number $i$
and ending in the sink with number $j$, $e\in\gamma$ are the edges of the path $\gamma$. The Lindstr\"om Lemma gives a beautiful formula
for minors of the matrix $M(\Gamma, w)$ in terms of weights $w$ \cite{FZ}:
$$
M(\Gamma, w)_{IJ} = \sum_{\gamma \in P\Gamma_{IJ}} \, \prod_{e \in \gamma} \, w(e).
$$
Here $I=\{ i_1 < i_2 < \dots < i_k\} $ and $J=\{ j_1 < j_2 <\dots < j_k\} $ are multi-indices of cardinality  $k=| I | = |J |$, $P\Gamma_{IJ}$ is the set of $k$--paths in $\Gamma$ starting in the sources with labels in $I$ and ending in the sinks with labels in $J$, and a $k$--path is a collection of $k$ paths with no common vertices. Note that all the minors are polynomials in the weights $w(e)$ for $e\in E\Gamma$.

For a network $\Gamma$ of type $n$, we introduce a family of subnetworks $\Gamma^{(k)}$ for 
$k=1, \dots, n$ such that $\Gamma^{(n)}=\Gamma$ and $\Gamma^{(k)} \subset \Gamma$ is the subnetwork of type $k$ which contains the last $k$ sources on $L$ and the last $k$ sinks on $R$.  The remaining sources and sinks of $\Gamma$ and the edges attached to them are deleted. 

\begin{example}
 Let $n=3$. Consider a network represented on the Figure \ref{ex3net} (note that the weights equal to $1$ are omitted in pictorial presentation).

\begin{figure}[!htbp]  
\begin{center}
\includegraphics[width=6cm]{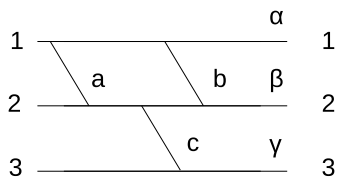}
\caption{An example of a weighted network for $n=3$}
\label{ex3net}
\end{center}
\end{figure}

 The matrix associated to it reads:
$$
M(\Gamma,w)=
\left(
\begin{array}{lll}
\alpha  & (a+b)\beta & ac\gamma \\
0 & \beta & c\gamma \\
0 & 0 & \gamma
\end{array}
\right)
$$

\end{example}

The weights of the planar network $\Gamma_s$ (see Fig. \ref{network2}) 
define a coordinate system on an open dense subset in $B_+$ \cite{FZ}. 

\begin{figure}[!htbp]  
\begin{center}
\includegraphics[width=8cm]{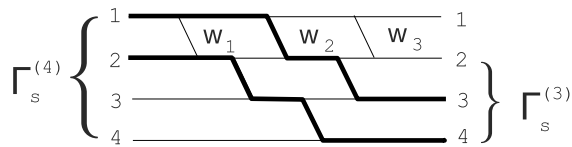}
\caption{network $\Gamma_s$ for $n=4$ and the minor $\Delta^{(4)}_2$}
\label{network2}
\end{center}
\end{figure}

\no By Lindstr\"om Lemma the minors $\Delta^{(k)}_l$ are monomials in terms of the weights. Moreover, the following proposition takes place:

\begin{prop}
The weights of the network $\Gamma_s$ are Laurent monomials in $\Delta^{(k)}_l$. 
\end{prop}

\begin{proof}
One can prove this claim by induction. For $n=2$ the statement is obvious. Assume that it holds for a certain $n$. We need to show that it also holds for $n+1$. By assumption, we already know that the weights  on the subnetwork of size $n$ are Laurent monomials in $\Delta$'s, and it remains to determine $n$ weights corresponding to the slanted edges of the upper floor of the network. Starting with the leftmost slanted edge, we notice that $\Delta^{(n+1)}_1$ is a product of $w_1$ and some weights from the lower subnetwork, $\Delta^{(n+1)}_2$ is a product of $w_2$ and some weights from the lower subnetwork {\em etc.} which proves the claim (see Figure \ref{network2} for illustration of the reasoning for $n=3$).
\end{proof}

\section{Tropical Gelfand-Zeitlin map}   \label{appC}

The Gelfand-Zeitlin cone in $\R^{n(n+1)/2}$ is defined in terms of coordinates
$\lambda^{(k)}_l$ with $n \geq k \geq l \geq 1$ by the interlacing inequalities
\be \label{eq:gz}
\lambda^{(k)}_l \geq \lambda^{(k-1)}_l \geq \lambda^{(k)}_{l+1}.
\ee
These inequalities are verified by the ordered eigenvalues of a Hermitian matrix
together with its principal submatrices (see \cite{HJ}). Let
$$
\zeta^{(k)}_l = \lambda^{(k)}_1 + \dots + \lambda^{(k)}_l
$$
for $k=1, \dots, n$ and $l=1, \dots, k$ and put $\zeta^{(k)}_0=0$ for all $k$.
Then,   \eqref{eq:gz}  is equivalent to the following system of inequalities,
\be \label{eq:gz2}
\begin{array}{lll}
\zeta^{(k)}_l + \zeta^{(k-1)}_{l-1}  & \geq & \zeta^{(k)}_{l-1} + \zeta^{(k-1)}_l, \\
\zeta^{(k)}_l + \zeta^{(k-1)}_{l}  & \geq & \zeta^{(k)}_{l+1} + \zeta^{(k-1)}_{l-1}
\end{array}
\ee
for $k=2, \dots, n$ and $l=1, \dots, k-1$. These inequalities can be visualized
as shown on Figure~\ref{triang2}. The variables $\zeta^{(k)}_l$ are placed in the vertices of the graph, and inequalities correspond to rhombi of two orientations. For each rhombus in this family, the sum of variables on the short diagonal is greater or equal to the sum of variables on the long diagonal.

\begin{figure}[!htbp]
\begin{center}
\includegraphics[width=8cm]{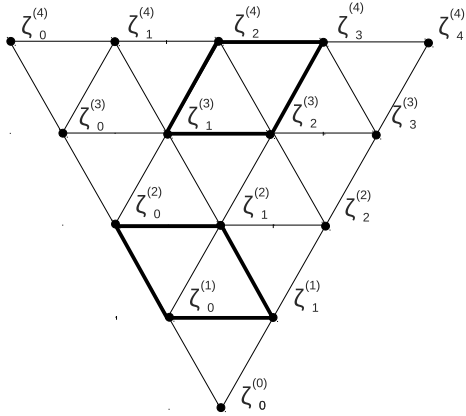}
\caption{Rhombi corresponding to Gelfand-Zeitlin inequalities  }
\label{triang2}
\end{center}
\end{figure}

Let $\Gamma$ be a planar network of type $n$ equipped with real weights
$w: E\Gamma \to \R$. Define a map $l: \R^{|E\Gamma|} \to \R^{n}$ as follows:
$$
l_i= {\rm max}_{\gamma \in P\Gamma_i} \, \sum_{e \in \gamma} w(e).
$$
For a network $\Gamma$, let $\Gamma^{(k)}$ be a subnetwok of type
$k$ obtained from $\Gamma$ by deleting the sources and sinks with
numbers $1, 2, \dots, n-k$ and the edges starting and ending in these vertices.
Define functions $l^{(k)}_i$ with $n \geq k \geq i \geq 1$ by applying
the functions $l_i$ to the weights of subnetworks $\Gamma^{(k)}$, that is 
$$
l_i^{(k)}= {\rm max}_{\gamma \in P\Gamma_i^{(k)}} \, \sum_{e \in \gamma} w(e).
$$

\no Then Theorem 2 in \cite{aps} states that the image of the combined map 
$l^{(k)}_i$ (the tropical Gelfand-Zeitlin map) is always contained 
in the Gelfand-Zeitlin cone in the form \eqref{eq:gz2}. That is, the functions
$l^{(k)}_i$ verify the inequalities
$$
\begin{array}{lll}
l^{(k)}_l + l^{(k-1)}_{l-1}  & \geq & l^{(k)}_{l-1} + l^{(k-1)}_l, \\
l^{(k)}_l + l^{(k-1)}_{l}  & \geq & l^{(k)}_{l+1} + l^{(k-1)}_{l-1} .
\end{array}
$$

Moreover, Theorem 3 in \cite{aps} states that for the planar network $\Gamma_s$ the image of the tropical Gelfand-Zeitlin map coincides with the Gelfand-Zeitlin cone. Furthermore, in this case the tropical Gelfand-Zeitlin map is a piece-wise linear map from $\R^{n(n+1)/2}$ to itself. Under this map, the space of weights splits into linearity chambers (on each chamber the tropical Gelfand-Zeitlin map is linear).It turns out that there is a unique principal linearity chamber $\C_0$ on which the Jacobian of the tropical Gelfand-Zeitlin map is non-vanishing, and it defines a bijection between $\C_0$ and $\C_{GZ}$. In particular, on $\C_0$ the maximum in the definition of $l^{(k)}_i$ is achieved on the multi-paths $\gamma^{(k)}_i$ of the type shown on Fig. \ref{paths} which are in one-to-one correspondence with the minors $\Delta^{(k)}_l$.

\begin{figure}[!htbp] 
\begin{center}
\includegraphics[width=12cm]{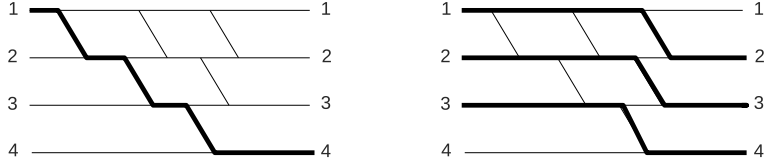}
\caption{Planar network $\Gamma^{(4)}_s$ and paths $\gamma_1^{(4)}$ and $\gamma^{(4)}_3$  }
\label{paths}
\end{center}
\end{figure}

The principal linearity chamber $\C_0$ admits the following pictorial description. 
Assign weights to connected components of the planar network $\Gamma$ according
to the following rule: for a region $\alpha$ add up weights of edges which bound $\alpha$
with sign $(+1)$ if the edge is  above or to the right of $\alpha$ and $(-1)$ if the edge
is below or to the left of $\alpha$, see Fig. \ref{prlin0}.

\begin{figure}[!htbp] 
\begin{center}
\includegraphics[width=5cm]{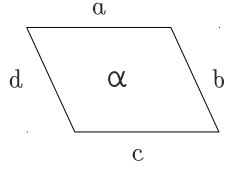}
\caption{$\omega_{\alpha}=a+b-c-d$}
\label{prlin0}
\end{center}
\end{figure}

The weighting of the planar network $\Gamma_s$ belongs to the principal chamber $\C_0$
(see Lemma 9 in \cite{aps}) if and only if the regions $\alpha^+_{k,l}$ have positive weight
and regions $\alpha^{-}_{k,l}$ have negative weight (see Fig. \ref{prlin}) for $k=2, \dots, n$ and
$l=1, \dots, k-1$.

\begin{figure}[!htbp] 
\begin{center}
\includegraphics[width=12cm]{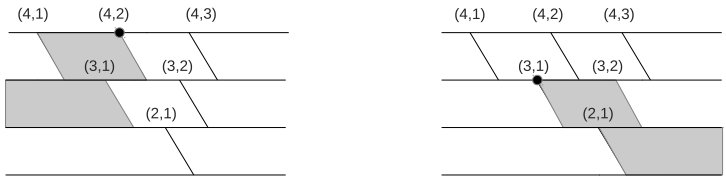}
\caption{$\alpha^{+}_{(4,2)}$ and $\alpha^{-}_{(3,1)}$}
\label{prlin}
\end{center}
\end{figure}

\end{document}